\documentclass[a4paper, 11pt]{article}
\usepackage{amsmath, amssymb,amscd, amsthm}
\usepackage[mathscr]{eucal}
\usepackage{graphics}
\usepackage{fullpage}
\usepackage{url}
\newcommand\cyr{%
 \renewcommand\rmdefault{wncyr}%
 \renewcommand\sfdefault{wncyss}%
 \renewcommand\encodingdefault{OT2}%
\normalfont\selectfont} \DeclareTextFontCommand{\textcyr}{\cyr}

\newtheorem{theorem}{Theorem}
\newtheorem{lemma}[theorem]{Lemma}
\newtheorem{corollary}[theorem]{Corollary}

\begin{document}

\title{\textbf{An Exact Asymptotic for the Square Variation of Partial Sum Processes}}

\author{Allison Lewko \thanks{Supported by a National Defense Science and Engineering Graduate Fellowship} \and Mark Lewko}

\date{}
\maketitle
\begin{abstract}We establish an exact asymptotic formula for the square variation of certain partial sum processes. Let $\{X_{i}\}$ be a sequence of independent, identically distributed mean zero random variables with finite variance $\sigma$ and satisfying a moment condition $\mathbb{E}\left[ |X_{i}|^{2+\delta} \right] < \infty$ for some $\delta > 0$. If we let $\mathcal{P}_{N}$ denote the set of all possible partitions of the interval $[N]$ into subintervals, then we have that $\max_{\pi \in \mathcal{P}_{N}} \sum_{I \in \pi } | \sum_{i\in I} X_{i}|^2 \sim 2 \sigma^2N \ln \ln(N)$ holds almost surely. This can be viewed as a variational strengthening of the law of the iterated logarithm and refines results of J. Qian on partial sum and empirical processes. When $\delta = 0$, we obtain a weaker `in probability' version of the result.
\end{abstract}

\section{Introduction}

Let $\{X_{i}\}$ be a sequence of independent, identically distributed random variables with mean $\mu < \infty$. The strong law of large numbers asserts that
\[ \sum_{i=1}^{N}X_{i} \sim N\mu \]
almost surely. Without loss of generality, one can assume that $X_{i}$ are mean zero by considering $X_{i}-\mu$. If we further assume a finite variance, that is $\mathbb{E}\left[|X_{i}|^2 \right] = \sigma^2 < \infty$, the Hartman-Wintner law of the iterated logarithm \cite{HW} gives an exact error estimate for the strong law of large numbers. More precisely,
\begin{equation}
\label{lil}  \left|\sum_{i=1}^{N}X_{i} \right|^2 \leq  (2+o(1))\sigma^2 N \ln\ln (N),
\end{equation}
holds a.s., where the constant $2$ cannot be replaced by a smaller constant. That is, the quantity $\sum_{i=1}^{N}X_{i}$ gets as large/small as $\pm \sqrt{ (2-\epsilon) \sigma N \ln\ln (N)}$ infinitely often (a.s.).

The purpose of our current work is to prove a more delicate variational asymptotic that refines the law of the iterated logarithm and captures more subtle information about the oscillations of sums of i.i.d random variables about their expected values. More precisely, we prove the following theorem. We let $\mathcal{P}_{N}$ denote the set of all possible partitions of the interval $[N]$ into subintervals, and we consider each $\pi \in \mathcal{P}_N$ as a collection of disjoint intervals, the union of which is $[N]$. We write $I \in \pi$ to denote that the interval $I \subseteq [N]$ belongs to $\pi$.

\begin{theorem}\label{thm:ourtheorem} Let $\{X_{i}\}$ be a sequence of independent, identically distributed mean zero random variables with variance $\sigma$ and satisfying $\mathbb{E}\left[|X_{i}|^{2+\delta}\right] < \infty$ for some $\delta >0$. Then we have almost surely:

\begin{equation}\label{mainasy}
\max_{\pi \in \mathcal{P}_{N}} \sum_{I \in \pi } \left| \sum_{i\in I} X_{i}\right|^2 \sim 2 \sigma^2N \ln \ln(N).
\end{equation}

\end{theorem}

Choosing the partition $\pi$ to contain a single interval $J=[1,N]$ immediately recovers (\ref{lil}), the upper bound in the law of the iterated logarithm.

It is unclear if the assumption $\mathbb{E}\left[|X_{i}|^{2+\delta}\right] < \infty$ can be removed. The analysis here should be able to be pushed further to handle a condition of the form $\mathbb{E}\left[|X_i|^{2}\varphi(X_i) \right]<\infty$ where $\varphi(x)$ is a positive increasing function that grows slower than $|x|^{\delta}$ for any $\delta>0$ (this requires analogous refinements of Lemma \ref{lem:Rosenthal} and Lemma \ref{lem:berry-esseen}), however this will not extend to the case when $\varphi(x)$ is bounded. Without an auxiliary moment condition, we are able to establish the following weaker `in probability' result.

\begin{theorem}Let $\{X_i\}$ be a sequence of independent, identically distributed mean zero random variables with finite variance $\sigma$. We then have that
\[ \frac{\max_{\pi \in \mathcal{P}_{N}} \sum_{I \in \pi } | \sum_{i\in I} X_{i}|^2}{2 \sigma^2 N \ln \ln(N)} \xrightarrow{p} 1.\]
\end{theorem}
\begin{flushleft}Here, $\xrightarrow{p}$ denotes convergence in probability.
\end{flushleft}

For the rest of the paper, we will denote $\{X_{i}\}_{i=1}^{N}$ more concisely as $\{X_{i}\}^{N}$ and write $||\{X_{i}\}^{N}||_{V^2}$ for the quantity $\sqrt{\max_{\pi \in \mathcal{P}_{N}} \sum_{I \in \pi } | \sum_{i\in I} X_{i}|^2}$. As indicated by the notation, this expression satisfies the triangle inequality: $||\{X_{i}+Y_{i}\}^{N}||_{V^2} \leq ||\{X_{i}\}^{N}||_{V^2} + ||\{Y_{i}\}^{N}||_{V^2}$. (This can be easily verified.)

\subsection{Previous Work}

In \cite{T}, Taylor proved an almost sure asymptotic for the path variation of Brownian motion. This is closely related to our main result in the case that $X_{i}$ are normally distributed. The question of the asymptotic order of $||\{X_{i}\}^{N}||_{V^2}$ in the case of more general $X_{i}$ is implicit in work of Dudley \cite{Dud} and Bretagnolle \cite{B}, where related questions about the $p$-variation of processes are studied. Dudley's interest in these $p$-variation norms stems from the fact that they majorize the (more classical) sup norm, but in many cases have nicer differentiability properties.

The most recent result we are aware of concerning our specific question appears in the work of J. Qian \cite{Q98}. There it is shown that

\begin{theorem}\label{thm:Qianlower} Let $\{X_i\}_{i=1}^\infty$ be i.i.d mean zero random variables with variance $\sigma$. Then for some constant $c$ we have that $\mathbb{P}\left[  \max_{\pi \in \mathcal{P}_{N}} \sum_{I \in \pi } | \sum_{i\in I} X_{i}|^2 \leq c \sigma^2 N \ln \ln(N)  \right] \rightarrow 0$ as $N \rightarrow \infty$.
\end{theorem}

\begin{theorem}\label{thm:Qianupper} Let $\{X_i\}_{i=1}^\infty$ be i.i.d mean zero random variables with variance $\sigma$ and $\mathbb{E}\left[|X_{i}|^{2+\delta}\right] < \infty$ for some $\delta >0$. For some constant $c'$ we have that \\ $\mathbb{P}\left[  \max_{\pi \in \mathcal{P}_{N}} \sum_{I \in \pi } | \sum_{i\in I} X_{i}|^2 \geq c' \sigma^2 N \ln \ln(N)  \right] \rightarrow 0$ as $N \rightarrow \infty$.
\end{theorem}
These results have been used in \cite{J} to show that that certain
variation operators which generalize (and majorize) classical maximal
operators arising in harmonic analysis are unbounded on certain $L^p$ spaces.

Notice that Qian's upper bound does not require a moment condition with $\delta >0$, but her lower bound does. Our results improve on these by establishing the exact constants as well as improving convergence in probability to almost sure convergence when $\delta >0$. When only the second moment is finite, we obtain convergence in probability results with the exact asymptotic constant.

\subsection{Notation} In our proofs below, we will often fix a positive integer $N$ and use $[N]$ to denote the set of integers $\{1, \ldots, N\}$. When we say $I$ is a subinterval of $[N]$, we mean that $I = (i_s, i_e]$ for some real numbers $0 \leq i_s \leq i_e \leq N$. We denote the length of the interval $I$ by $|I|$ (i.e. $|I| = i_e-i_s$). When $i_s$ and $i_e$ are integers, this quantity is equal to the number of positive integers contained in $I$. When we say $I' \subseteq I$ ($I'$ is a subinterval of $I$), we mean that $I' = (i'_s, \ldots, i'_e]$ for some real numbers $i'_s, i'_e$ satisfying $i_s \leq i'_s \leq i'_e \leq i_e$.

We will consider independent, identically distributed random variables $X_1, X_2, \ldots$. We will routinely use $S_\ell$ to denote the partial sum $S_\ell = X_1 + \cdots + X_\ell$. For an interval $I$, we use $S_I$ to denote $S_I:= \sum_{i \in I} X_i$ (i.e. the partial sum of the variables $X_i$ whose indices $i$ are contained in the interval $I$).
We use $\ln$ to denote the natural logarithm and $\log$ to denote the base 2 logarithm. We use $exp(x)$ as an alternate notation for $e^x$.

In our proofs, we will often refer to ``constants" whose values depend on $\delta$ (and only on $\delta$). We will often not reflect this dependence in our notation. Throughout our proofs, $\delta$ should be thought of as a fixed, positive constant.

\section{The Upper Bound}
In this section, we prove the following theorem:

\begin{theorem}\label{thm:ubound} We let $X_1, X_2, \ldots$ denote independent, identically distributed random variables with $\mathbb{E}[X_i] = 0$ and $Var[X_i]=1$. We further assume that $\mathbb{E}[|X_i|^{2(1+\delta)}] < \infty$, for some $\delta >0$. Then, for every $\epsilon > 0$,
\[\limsup_{N \rightarrow \infty} \frac{\left|\left| \{X_i\}^N \right| \right|^2_{V^2}}{N \ln \ln N} \leq 2(1+\epsilon) \emph{ a.s.} \]
\end{theorem}

The proof of this theorem will proceed in several stages. First, we will fix $\epsilon > 0$ and classify intervals $I \subseteq [N]$ into three disjoint categories, ``good," ``medium", and ``bad". (This same strategy is used in \cite{T}.) We say an interval $I$ is ``good" if
\begin{equation}\label{def:good}
S_I^2 \leq (2+ \epsilon) |I| \ln \ln N.
\end{equation}
We say it is ``medium" if
\begin{equation}\label{def:medium}
(2+ \epsilon) |I| \ln \ln N < S_I^2 \leq B|I| \ln \ln N,
\end{equation}
and it is ``bad" if
\begin{equation}\label{def:bad}
S_I^2 > B |I| \ln \ln N.
\end{equation}
The precise value of the parameter $B$ will be chosen later. For now, we simply think of it as a constant depending only on $\delta$.

We will deal with each class of intervals separately. We begin by considering the contribution of the bad intervals to the value of $\left| \left| \{X_i\}^N \right| \right|^2_{V^2}$.

\subsection{The Bad Intervals}
To suitably bound the contribution of the bad intervals, we will begin by essentially reducing the space of allowable partitions. We assume for simplicity that $N$ is a power of two, denoted by $N = 2^n$. We will later argue that our results extend to all positive integers $N$. When $N = 2^n$, we consider intervals of the form
\[((c-1)2^i, c2^i], \; i \in \{0, 1, \ldots, n\}, \; c \in \{1,\ldots, 2^{n-i}\}.\] This gives us $n+1 = \log(N)+1$ levels of intervals, where the $i^{th}$ level contains $2^{n-i}$ disjoint intervals, each of size $2^i$. We now augment this family of intervals by adding ``half shifts" of each level. More precisely, we also consider intervals of the form
\[((c-1)2^i + 2^{i-1}, c2^i + 2^{i-1}], \; i \in \{1, \ldots, n-1\}, \; c \in \{1,\ldots, 2^{n-i}-1\}.\]
This approximately doubles our total number of intervals. We will call the first family of intervals $\mathcal{F}$, and the second family of intervals $\mathcal{F}_{s}$. $\mathcal{F}$ includes $n+1$ levels of intervals, while $\mathcal{F}_{s}$ includes $n-1$ levels of intervals. Within each family at each level, the intervals are disjoint.

We now show:
\begin{lemma}\label{lem:intervals} Let $I' \subseteq [N]$ denote an arbitrary interval. There exists some interval $I \in \mathcal{F} \cup \mathcal{F}_{s}$ such that $I' \subseteq I$ and $|I| < 4|I'|$.
\end{lemma}

\begin{proof} We define $i$ to be the non-negative integer satisfying $2^{i-1} < |I'| \leq 2^i$. If $i \geq n-1$, then $I:= [1, \ldots, N] \in \mathcal{F}$ suffices. For $i < n-1$, we consider the intervals in $\mathcal{F} \cup \mathcal{F}_{s}$ of length $2^{i+1}$. There are two cases: either $I'$ is contained in some $I \in \mathcal{F}$ of size $2^{i+1}$ (and therefore, we are done), or $I'$ must contain a right endpoint of some $I' \in \mathcal{F}$ of size $2^{i+1}$. In other words, $I'$ contains $c2^{i+1}$ for some $c \in \{1, \ldots, 2^{n-i-1}-1\}$. Since $|I'| \leq 2^i$, this implies that $I' \subseteq (c2^{i+1}- 2^i, c2^{i+1} + 2^i)$. We can alternatively express this as:
\[I' \subseteq ((c-1)2^{i+1}+ 2^i, c2^{i+1} + 2^i] \in \mathcal{F}_{s}.\]
So in both cases, we have an $I \in \mathcal{F} \cup \mathcal{F}_{s}$ such that $I' \subseteq I$ and $|I| < 4|I'|$.
\end{proof}

For the purpose of bounding the contribution of bad intervals, this allows us to consider only intervals in $\mathcal{F} \cup \mathcal{F}_{s}$ (to some extent). More precisely, for each interval $I \in \mathcal{F} \cup \mathcal{F}_{s}$, we will consider the random variable \[\max_{I' \subseteq I} S_{I'}^2.\] If some $I' \subseteq I$ of size $|I'| > \frac{1}{4}|I|$ is bad, meaning that $|S_{I'}|^2 > B |I'| \ln \ln N$, then
\[\max_{I' \subseteq I} S_{I'}^2 > \frac{B}{4} |I| \ln \ln N.\] To enable us to later consider values of $N$ which are not powers of two, we actually consider ``badness" with respect to $N/2$ instead of $N$. More precisely, we let $\mathbb{I}_{Bad, I}$ denote the indicator variable of the event
\[\max_{I' \subseteq I} S_{I'}^2 > \frac{B}{8} |I| \ln \ln N.\]
for a particular interval $I$. Here we have used the very loose bound that $\ln \ln (N/2)\geq \frac{1}{2} \ln \ln N$, for $N \geq 4$. Note that $I$ can contain a subinterval of size $> \frac{1}{4}|I|$ which is bad with respect to $N/2$ (or anything between $N/2$ and $N$) only when this event occurs.

It is then clear that the contribution of the bad intervals to the value of $\left| \left| \{X_i\}^N \right| \right|^2_{V^2}$ is upper bounded by:
\begin{equation}\label{badcontribution0}
3 \sum_{I \in \mathcal{F} \cup \mathcal{F}_{s}} \max_{I' \subseteq I} S_{I'}^2 \cdot \mathbb{I}_{Bad, I}.
\end{equation}
To see this, note that each bad interval $I'$ in the partition achieving the maximal value is contained in some $I \in \mathcal{F} \cup \mathcal{F}_{s}$ such that $\mathbb{I}_{Bad,I} = 1$. Since the intervals in the maximal partition must be disjoint, each such $I$ will only be associated with at most $3$ $I'$'s. Thus, to control the contribution of the bad intervals, it suffices to prove a suitable upper bound on (\ref{badcontribution0}) that holds almost surely. As a shorthand notation, we define the variable $Y_I := \max_{I' \subseteq I} S_{I'}^2$. Then the quantity we need to bound can be written a bit more succintly as:
\begin{equation} \label{badcontribution}
3 \sum_{I \in \mathcal{F} \cup \mathcal{F}_{s}} Y_I \cdot \mathbb{I}_{Bad, I}.
\end{equation}

To bound (\ref{badcontribution}), we will rely on several standard lemmas.

%
%
%

\begin{lemma}\label{lem:etemadi} (Etemadi's Inequality - Theorem 1 in \cite{E}) Let $X_1, X_2, \ldots, X_k$ denote independent random variables and let $a > 0$. Let $S_\ell := X_1 + \cdots + X_\ell$ denote the partial sum. Then:
\[\mathbb{P} [\max_{1 \leq \ell \leq k} |S_\ell| \geq 3 a] \leq 3 \max_{1 \leq \ell \leq k} \mathbb{P}[|S_\ell| \geq a].\]
\end{lemma}

\begin{lemma}\label{lem:Doob1}(Doob) Let $\{M_i\}_{i=1}^L$ be a submartingale taking nonnegative real values, and $p >1$. Then:
\[\mathbb{E}\left[\left(\max_{1 \leq \ell \leq L} M_\ell\right)^p\right] \leq \left(\frac{p}{p-1}\right)^p \mathbb{E}[M_L^p].\]
\end{lemma}

\begin{lemma}\label{lem:Rosenthal} (Rosenthal's Inequality - Thm. 3 in \cite{R}) Let $2 < p < \infty$. Then there exists a constant $K_p$ depending only on $p$, so that if $X_1, \ldots, X_\ell$ are independent random variables with $\mathbb{E}[X_i] = 0$ for all $i$ and $\mathbb{E}[|X_i|^p] < \infty$ for all $i$, then:
\[\left(\mathbb{E}[|S_\ell|^p]\right)^{1/p} \leq K_p\; \max \left\{ \left( \sum_{i=1}^{\ell} \mathbb{E}[|X_i|^p]\right)^{1/p}, \left(\sum_{i=1}^{\ell} \mathbb{E}[|X_i|^2]\right)^{1/2}\right\}.\]
\end{lemma}

We now prove:

\begin{lemma}\label{lem:expectation} For any interval $I \in \mathcal{F} \cup \mathcal{F}_s$, \[\mathbb{E}[|Y_I|^{1+\delta}] \leq C |I|^{1+\delta},\] where $C$ is a constant depending only on $\delta$ .
\end{lemma}

\begin{proof} For each $I = (i_s, \ldots, i_e]$, we define the notation $S_{I,k} := X_{i_s+1} + \cdots X_{i_s + k}$ to denote the partial sum of the first $k$ variables in the interval $I$ (for values of $k$ from 1 to $|I|$).
We also define the random variable $\widetilde{Y}_I$ as:
\[\widetilde{Y}_I := \max_{1 \leq k \leq |I|} S_{I,k}^2.\]

We observe that $\mathbb{E}[|Y_I|^{1+\delta}] \leq 4^{1+\delta} \mathbb{E}[|\widetilde{Y}_I|^{1+\delta}]$. To see this, consider an arbitrary interval $I' = (i'_s, \ldots, i'_e] \subseteq I = (i_s, \ldots, i_e]$, where $i'_s \neq i_s$. We let $k_1$ denote the number of integers contained in $(i_s, i'_e]$ and we let $k_2$ denote the number of integers contained in $(i_s, i'_s]$. Then $S_{I,k_1} = S_{I'} + S_{I,k_2}$, so
\[|S_{I'}| \leq 2 \max \left\{ |S_{I,k_1}|, |S_{I,k_2}|\right\}.\]
This implies $\mathbb{E}[|Y_I|^{1+\delta}] \leq 4^{1+\delta} \mathbb{E}[|\widetilde{Y}_I|^{1+\delta}]$. (In fact, it implies the stronger fact that $Y_I \leq 4\widetilde{Y}_I$ always holds, and we will use this again later.)

Now, $\{S_{I,k}\}$ is a martingale, so $\{|S_{I,k}|\}$ is a submartingale (by Jensen's inequality). Thus, by Lemma \ref{lem:Doob1}:
\begin{equation}\label{fromdoob}
\mathbb{E}[|\widetilde{Y}_I|^{1+\delta}] = \mathbb{E}\left[ \left(\max_{1 \leq k \leq |I|} |S_{I,k}|\right)^{2(1+\delta)}\right] \leq C' \mathbb{E}[|S_I|^{2(1+\delta)}].
\end{equation}
for some constant $C'$ depending only on $\delta$.

Applying Lemma \ref{lem:Rosenthal}, we see that:
\begin{equation}\label{fromRosenthal}
\mathbb{E}[|S_I|^{2(1+\delta)}] \leq K'|I|^{1+\delta},
\end{equation}
for some constant $K'$ depending only on $\delta$.
Combining (\ref{fromRosenthal}) with (\ref{fromdoob}) (and recalling that $Y_I \leq 4\widetilde{Y}_I$), we have shown:
\[\mathbb{E}[|Y_I|^{1+\delta}] \leq C |I|^{1+\delta}\]
for some constant $C$ depending only on $\delta$.
\end{proof}

Our next goal is to derive a suitable upper bound on the quantity $\mathbb{E}[Y_I \mathbb{I}_{Bad, I}]$ for every interval $I$. To do this, we will need one more standard lemma.

\begin{lemma}\label{lem:berry-esseen} (Berry-Esseen Theorem\footnote{This can be found in \cite{CT}, p.322, for example.}) Let $X_1, X_2, \ldots$ be independent random variables with $\mathbb{E}[X_i] = 0$, $\mathbb{E}[X_i^2] = 1$, and $\mathbb{E}[|X_i|^{2+\gamma}] \leq M$ for all $i$, and some $\gamma \in (0,1]$. Then there exists a universal constant $C_\gamma$ such that for all positive integers $k$:
\[\sup_{-\infty < x < \infty} \left| \mathbb{P}\left[\frac{S_k}{\sqrt{k}} < x\right] - \frac{1}{\sqrt{2\pi}} \int_{-\infty}^x e^{-y^2/2} dy \right| \leq C_\gamma \left( \frac{M}{k^{\gamma/2}}\right).\]
\end{lemma}

To upper bound $\mathbb{E}[Y_I \mathbb{I}_{Bad, I}]$, we begin by applying H\"{o}lder's Inequality with $p := 1+\delta$ and $q$ defined so that $\frac{1}{p} + \frac{1}{q} = 1$. This gives us:
\[\mathbb{E}[Y_I \mathbb{I}_{Bad,I}] = \mathbb{E}[|Y_I \mathbb{I}_{Bad,I}|] \leq \left(\mathbb{E}[|Y_I|^{1+\delta}]\right)^{1/(1+\delta)} \left(\mathbb{E}[|\mathbb{I}_{Bad,I}|^q]\right)^{1/q}\]
Applying Lemma \ref{lem:expectation}, we see that
\[\left(\mathbb{E}[|Y_I|^{1+\delta}]\right)^{1/(1+\delta)} \leq C^{1/(1+\delta)} |I|.\]
Since $\mathbb{I}_{Bad,I}$ only takes values in $\{0,1\}$, we also have
\[\mathbb{E}[|\mathbb{I}_{Bad,I}|^q] = \mathbb{E}[\mathbb{I}_{Bad,I}] = \mathbb{P}[\mathbb{I}_{Bad,I} = 1].\]

We now consider
\[\mathbb{P}\left[\mathbb{I}_{Bad,I} = 1\right] = \mathbb{P}\left[Y_I > \frac{B}{8} |I| \ln \ln N\right].\]
We recall the definition of $\widetilde{Y}_I$ and the fact that $Y_I \leq 4 \widetilde{Y}_I$ from the proof of Lemma \ref{lem:expectation}. We then have:
\[\mathbb{P}\left[Y_I > \frac{B}{8} |I| \ln \ln N\right] \leq \mathbb{P}\left[\widetilde{Y}_I > \frac{B}{32} |I| \ln \ln N\right] = \mathbb{P}\left[ \max_{1 \leq k \leq |I|} |S_{I,k}| > \sqrt{\frac{B}{32}} |I|^{1/2} (\ln \ln N)^{1/2}\right].\]

By Lemma \ref{lem:etemadi}, this quantity is
\[ \leq 3 \max_{1 \leq k \leq |I|} \mathbb{P}\left[ |S_{I,k}| \geq \frac{\sqrt{B}}{12 \sqrt{2}}|I|^{1/2} (\ln \ln N)^{1/2}\right] .\]
We will bound this probability using Chebyshev's inequality for values of $k$ which are $< |I|^{1/2}$, and using Lemma \ref{lem:berry-esseen} for larger values of $k$.

For $k < |I|^{1/2}$, we apply Chebyshev's inequality to obtain:
\[ \mathbb{P}\left[ |S_{I,k}| \geq \frac{\sqrt{B}}{12\sqrt{2}}|I|^{1/2} (\ln \ln N)^{1/2}\right] \leq \frac{288 \mathbb{E}[|S_{I,k}|^2]}{B|I|\ln\ln N}.\]
We note that $\mathbb{E}[|S_{I,k}|^2] = \mathbb{E}[S_{I,k}^2] = k$ (recall that $S_{I,k}$ is a sum of $k$ independent random variables, each with mean 0 and variance 1. Since $k < |I|^{1/2}$, this gives us:
\[\mathbb{P}\left[ |S_{I,k}| \geq \frac{\sqrt{B}}{12\sqrt{2}}|I|^{1/2} (\ln \ln N)^{1/2}\right] \leq \frac{288}{B|I|^{1/2} \ln \ln N}.\]

For $|I|^{1/2} \leq k \leq |I|$, we apply Lemma \ref{lem:berry-esseen} to obtain:
\[ \mathbb{P}\left[|S_{I,k}| \geq \frac{\sqrt{B}}{12\sqrt{2}}|I|^{1/2} (\ln \ln N)^{1/2}\right] \leq \mathbb{P}\left[\frac{|S_{I,k}|}{\sqrt{k}} > \frac{\sqrt{B \ln \ln N}}{12\sqrt{2}}\right]\]
\[ \leq \frac{1}{\sqrt{2\pi}} \int_{\frac{\sqrt{B\ln \ln N}}{12\sqrt{2}}}^\infty e^{-y^2/2} dy + \frac{D}{k^{\delta}}\]
for some constant $D$ depending on $\delta$ (we are applying the lemma with $\gamma = 2\delta$, and $\mathbb{E}[|X_i|^{2+2\delta}]$ is a constant).

To bound the integral, we proceed as follows (assuming that $N$ is large enough so that $\frac{\sqrt{B\ln \ln N}}{12} \geq 1$):
\[\int_{\frac{\sqrt{B\ln \ln N}}{12\sqrt{2}}}^\infty e^{-y^2/2} dy \leq \int_{\frac{\sqrt{B \ln \ln N}}{12\sqrt{2}}}^\infty y e^{-y^2/2} dy = \left.-e^{-y^2/2}\right]_{\frac{\sqrt{B\ln\ln N}}{12\sqrt{2}}}^\infty = \left(\frac{1}{\ln N} \right)^{\frac{B}{576}}.\]
Thus, for $|I|^{1/2} \leq k \leq |I|$, we have shown:
\[ \mathbb{P}\left[|S_{I,k}| \geq \frac{\sqrt{B}}{12\sqrt{2}}|I|^{1/2} (\ln \ln N)^{1/2}\right] \leq \frac{1}{\sqrt{2\pi}} \left(\frac{1}{\ln N} \right)^{\frac{B}{576}} + \frac{D}{|I|^{\delta/2}}.\]

We define $\sigma = \min \{1/2, \delta/2\}$. Then, for some constant $D'$ depending on $\delta$ and for all $k$, we have:
\[\mathbb{P}\left[|S_{I,k}| \geq \frac{\sqrt{B}}{12}|I|^{1/2} (\ln \ln N)^{1/2}\right] \leq \frac{1}{\sqrt{2\pi}} \left(\frac{1}{\ln N} \right)^{\frac{B}{576}} + \frac{D'}{|I|^{\sigma}}.\]
Thus,
\[\mathbb{P}[\mathbb{I}_{Bad,I} = 1] \leq \frac{3}{\sqrt{2\pi}} \left(\frac{1}{\ln N} \right)^{\frac{B}{288}} + \frac{3D'}{|I|^{\sigma}}.\]

Putting everything together, we have that
\begin{equation}\label{expectedbad}
\mathbb{E}[Y_I \mathbb{I}_{Bad,I}] \leq C^{1/(1+\delta)} |I|\left(\frac{3}{\sqrt{2\pi}} \left(\frac{1}{\ln N} \right)^{\frac{B}{576}} + \frac{3D'}{|I|^{\sigma}}\right)^{1- \frac{1}{1+\delta}}.
\end{equation}

Next, we show that the contribution of intervals $I$ satisfying $|I| \geq (\log(N))^d$ to the quantity (\ref{badcontribution}) is not too large, where we define the parameter $d:= \frac{2}{\sigma \left(1 - \frac{1}{1+\delta}\right)}$. For this, we will use (\ref{expectedbad}) and Kronecker's lemma.

\begin{lemma}\label{lem:kronecker} (Kronecker's Lemma) Let $a_1, a_2, \ldots$ be a sequence of real numbers such that $a_1 \leq a_2 \leq a_3 \ldots$ and $a_j \rightarrow \infty$ as $j\rightarrow \infty$. Then if $x_1, x_2, \ldots$ is a sequence of real numbers such that $\sum_{j=1} x_j /a_j$ converges,
\[a_k^{-1} \sum_{j=1}^k x_j \rightarrow 0.\]
\end{lemma}

We now prove:
\begin{lemma}\label{lem:bigintervals} Let $m$ denote a positive integer. Then:
\[\frac{1}{2^m \ln\ln (2^m)} \sum_{\substack{I \in \mathcal{F} \cup \mathcal{F}_s \\ |I| \geq (\log(2^m))^d}} Y_I \mathbb{I}_{Bad,I} \rightarrow 0 \text{  a.s. as } m \rightarrow \infty,\]
where the indicator variable $\mathbb{I}_{Bad,I}$ and $\mathcal{F}\cup \mathcal{F}_s$ are defined with respect to $N = 2^m$.
\end{lemma}

\begin{proof} For each interval $I \in \mathcal{F} \cup \mathcal{F}_s$, there is a minimal value of $n$ such that $I \subseteq [2^n]$.
We order our sum over the $I$'s according to their associated values of $n$ (and otherwise arbitrarily): i.e. we first sum terms for $I$'s with $n=2$, then with $n=3$, and so on, and we only include those $I$'s satisfying $|I| \geq (\log(2^n))^d$. (We will ignore the very small number of terms with $n=1$ for convenience.)
We let $I_1, I_2, I_3, \ldots $ denote the resulting ordered sequence of all intervals $I$ which are contained in $\mathcal{F} \cup \mathcal{F}_s$ for some value of $N$ (which is a power of 2) and also satisfy $|I| \geq (\log(2^n))^d$. For each such $I_i$, we define
\[a_i := 2^n \ln \ln (2^n),\]
where $n$ is defined from $I_i$ as above. Since we consider $N$ going to infinity, we get an infinite sequence of $I$'s. For a fixed $N$, we let $I_{N_e}$ denote the final interval $I \subseteq [N]$ appearing in the infinite sequence.

We also define a new indicator variable, $\mathbb{I}_{Bad,I,n}$, which indicates the event that an interval $I$ is ``bad" with respect to the value $N := 2^n$ (where $n$ is defined from $I$ as above). Note that when $\mathbb{I}_{Bad,I}$ is the indicator for $I$ being ``bad" with respect to a larger $N$, then $\mathbb{I}_{Bad,I} = 1$ implies that $\mathbb{I}_{Bad,I,n} = 1$ as well. We then have (for any $N$ which is a power of 2):
\[\frac{1}{N\ln \ln N} \sum_{\substack{I \in \mathcal{F} \cup \mathcal{F}_s \\ |I| \geq (\log(N))^d}} Y_I \mathbb{I}_{Bad,I} \leq
\frac{1}{a_{N_e}} \sum_{i=1}^{N_e} Y_{I_i} \mathbb{I}_{Bad,I_i,n}.\]
It is a bit easier to work with the variables $Y_I \mathbb{I}_{Bad,I,n}$ than the original variables $Y_I \mathbb{I}_{Bad,I}$, since the latter \emph{depend} on the value of $N$, and the former do not. Hence we can now think of $N$ going to infinity just in terms of adding more random variables to our sum, instead of needing to change the definition of all the random variables with each change of $N$.

Thus, it suffices for us to prove that:
\[\lim_{M \rightarrow \infty} \frac{1}{a_M} \sum_{i=1}^{M} Y_{I_i} \mathbb{I}_{Bad,I_i,n} = 0 \text{ a.s.}.\]
By Lemma \ref{lem:kronecker} (note that $a_M \rightarrow \infty$ as $M \rightarrow \infty$), this follows if the sum
\[\sum_{i=1}^{\infty} \frac{1}{a_i} Y_{I_i} \mathbb{I}_{Bad,I_i,n}\]
converges almost surely. This in turn follows if:
\begin{equation}\label{finite}
\sum_{i=1}^{\infty} \frac{1}{a_i} \mathbb{E}[|Y_{I_i} \mathbb{I}_{Bad,I_i,n}|] < \infty.
\end{equation}

For a fixed $n$, we will have contributions from intervals of size $2^j$ for values of $j$ ranging from $d\log n$ to $n$. Note that since we only consider values of $n\geq 2$, we will have $j\geq 1$. By (\ref{expectedbad}), the sum of the expectations $\mathbb{E}[|Y_{I} \mathbb{I}_{Bad,I,n}|]$ for intervals $|I| = 2^j$ with the value of $n$ is at most:
\[C' \cdot 2^n \cdot  \left( \left(\frac{1}{n}\right)^{B/288}+ \frac{D''}{2^{\sigma j}}\right)^{1 - \frac{1}{1+\delta}}, \]
where $D''$ and $C'$ are constants depending only on $\delta$. We then see that (\ref{finite}) is dominated by:
\begin{equation}\label{finite2}
C' \sum_{n=2}^{\infty} \frac{1}{\ln n} \sum_{j=d \log n}^n  \left( \left(\frac{1}{n}\right)^{B/288}+ \frac{D''}{2^{\sigma j}}\right)^{1 - \frac{1}{1+\delta}}.
\end{equation}

We now note that for any positive real values $x, y, \gamma$, we have \[(x+y)^{\gamma} \leq (2 \max \{x, y\})^\gamma = 2^{\gamma} \max \{x^{\gamma}, y^{\gamma}\} \leq 2^{\gamma}(x^\gamma+ y^\gamma).\] Applying this to (\ref{finite2}), we see it is
\begin{equation}\label{finite3}
\leq C'' \sum_{n=2}^{\infty} \frac{1}{\ln n} \sum_{j= d\log n}^{n}\left( \left(\frac{1}{n}\right)^{\frac{B}{576} \left(1 - \frac{1}{1+\delta}\right)} + \frac{D'''}{2^{\sigma' j}}\right),
\end{equation}
where $C'', D'''$ and $\sigma'$ are constants depending on $\delta$. More specifically, $\sigma' = \sigma \left( 1- \frac{1}{1+\delta}\right)$.

We split this into two pieces, and first consider the sum:
\[D''' \sum_{n =2}^{\infty} \frac{1}{\ln n} \sum_{j= d\log n}^{n} \frac{1}{2^{\sigma' j}}.\]
We note that
\[\sum_{j= d\log n}^{\infty} \frac{1}{2^{\sigma' j}} \leq K_{\sigma'} \frac{1}{n^{d\sigma'}}\]
for some constant $K_{\sigma'}$ depending on $\sigma'$. Therefore,
\[\sum_{n =2}^{\infty} \frac{1}{\ln n} \sum_{j= d\log n}^{n} \frac{1}{2^{\sigma' j}} \leq K_{\sigma'} \sum_{n=2}^{\infty} \frac{1}{ n^{d\sigma'} \ln n}.\]
Since $d = \frac{2}{\sigma'}$, this sum converges.

Next we consider the sum
\[\sum_{n=2}^{\infty} \frac{1}{\ln n}\cdot  n^{-\frac{B}{576}\left(1 - \frac{1}{1+\delta}\right)} \sum_{j = d\log n}^{n} 1.\]
Since $\sum_{j=d\log n}^n 1 \leq n$, it suffices to consider
\[\sum_{n=2}^{\infty} \frac{1}{\ln n} \cdot n^{1-\frac{B}{576}\left(1 - \frac{1}{1+\delta}\right)}.\]

At this point, we choose the value of $B$ so that
\[\frac{B}{576}\left(1 - \frac{1}{1+\delta}\right)-1 >1.\]
This ensures that the sum converges, and the proof of the lemma is complete.
\end{proof}

To conclude our treatment of the bad intervals, we must also show that the contribution of intervals $I$ with $|I| < (\log(N))^d$ is not too large.
To do this, we return to considering a fixed value of $N$ (which is a power of 2) and the indicator variables $\mathbb{I}_{Bad,I}$ are all with respect to this $N$. For each $I \in \mathcal{F}\cup \mathcal{F}_s$, we define the random variable $Z_I := Y_I\mathbb{I}_{Bad,I}$. We first consider intervals $I \in \mathcal{F}$ of a fixed size $2^i < (\log(N))^d$. There are $L:= N \cdot 2^{-i}$ such intervals, and we denote the associated random variables $Z_I$ as $Z_1, \ldots, Z_L$. We prove:

\begin{lemma}\label{lem:errorprob}
\[\mathbb{P}\left[ \left| \sum_{j=1}^L Z_j - \sum_{j=1}^L \mathbb{E}[Z_j]\right| \geq L \right] \leq K (2^i)^{1+2\delta} N^{-\delta},\]
where $K$ is a constant depending only on $\delta$ and not on $i$ or $N$.
\end{lemma}

\begin{proof} We define $\overline{Z}_j := Z_j \mathbb{I}_{|Z_j|\leq L}$. In other words, $\overline{Z}_j$ is the truncation of the (non-negative) random variable $Z_j$ at the value $L$. We can then write:
\[\mathbb{P}\left[ \left| \sum_{j=1}^L Z_j - \sum_{j=1}^L \mathbb{E}[Z_j]\right| \geq L \right] =
\mathbb{P}\left[\left|\sum_{j=1}^L \left(\overline{Z}_j - \mathbb{E}[\overline{Z}_j]\right) + \sum_{j=1}^L \left(Z_j - \overline{Z_j} - \mathbb{E}[Z_j] + \mathbb{E}[\overline{Z}_j]\right) \right| \geq L\right].\]

We consider
\[\mathbb{E}[Z_j] -\mathbb{E}[\overline{Z}_j] = \int_{L}^{\infty} \mathbb{P}[Z_j>t] dt.\]
By Chebyshev's inequality,
\[\mathbb{P}[Z_j > t] \leq \frac{\mathbb{E}[|Z_j|^{1+\delta}]}{t^{1+\delta}} .\]
Inserting this into the integral, we obtain:
\[\left|\mathbb{E}[Z_j] - \mathbb{E}[\overline{Z}_j]\right| \leq \mathbb{E}[|Z_j|^{1+\delta}] \int_{L}^\infty t^{-1-\delta} dt = \frac{\mathbb{E}[|Z_j|^{1+\delta}]}{\delta L^{\delta}}.\]
We now see that:
\[\mathbb{P}\left[\left|\sum_{j=1}^L \left(\overline{Z}_j - \mathbb{E}[\overline{Z}_j]\right) + \sum_{j=1}^L \left(Z_j - \overline{Z_j} - \mathbb{E}[Z_j] + \mathbb{E}[\overline{Z}_j]\right) \right| \geq L\right]\]\[ \leq L \mathbb{P}[|Z_j|> L] + \mathbb{P}\left[ \left| \sum_{j=1}^L (\overline{Z}_j - \mathbb{E}[\overline{Z}_j])\right| \geq L - \frac{\mathbb{E}[|Z_j|^{1+\delta}]}{\delta L^{\delta}}\right].\]
Here, we have applied the union bound, the fact that the $Z_j$'s are identically distributed, and that $Z_j - \overline{Z}_j =0$ when $|Z_j| \leq L$.

We will bound these two quantities separately. First, by applying Chebyshev's inequality, we have:
\begin{equation}\label{simplebound}
L \mathbb{P}[|Z_j|> L] \leq L^{-\delta} \mathbb{E}[|Z_j|^{1+\delta}].
\end{equation}
To bound the second quantity, we let $L':= L - \frac{\mathbb{E}[|Z_j|^{1+\delta}]}{\delta L^{\delta}}$ and we let $\tilde{Z}_j := \overline{Z}_j -\mathbb{E}[\overline{Z}_j]$. By another application of Chebyshev's inequality,
\[\mathbb{P}\left[ \left| \sum_{j=1}^L \tilde{Z}_j \right| \geq L'\right] \leq \frac{\mathbb{E}\left[ \left(\sum_{j=1}^L \tilde{Z}_j\right)^2\right]}{(L')^2}.\]
Since the random variables $\tilde{Z}_j$ are independent and mean zero,
\[\mathbb{E}\left[ \left(\sum_{j=1}^L \tilde{Z}_j\right)^2\right] = L\mathbb{E}[\tilde{Z}_j^2].\]
We then observe:
\begin{equation}\label{int}
\mathbb{E}[\tilde{Z}_j^2] = 2\int_{0}^\infty x \mathbb{P}[|\tilde{Z}_j|>x] dx \leq 2+ 2\int_1^\infty x \mathbb{P}[|\tilde{Z}_j|>x] dx.
\end{equation}
We recall that $\tilde{Z}_j := \overline{Z}_j - \mathbb{E}[\overline{Z_j}]$. Since $\overline{Z}_j$ is a non-negative random variable which is $\leq L$, we have that $|\tilde{Z}_j| \leq L$ as well. Therefore, (\ref{int}) becomes:
\[ 2+ 2\int_1^L x \mathbb{P}[|\tilde{Z}_j|> x] dx \leq 2 + 2\int_1^L x \cdot \frac{\mathbb{E}[|\tilde{Z}_j|^{1+\delta}]}{x^{1+\delta}} =
2+ 2\mathbb{E}[|\tilde{Z}_j|^{1+\delta}] \int_1^L x^{-\delta} dx\]
\[ = 2+ \frac{2}{1-\delta}\mathbb{E}[|\tilde{Z}_j|^{1+\delta}] \left(L^{1-\delta}-1\right).\]

To put this all together and simplify our expressions, we recall that $L := N \cdot 2^{-i}$, where $2^i < (\log N)^d$. Thus, $L > N(\log N)^{-d}$. This is very large compared to the value of $\mathbb{E}[|Z_j|^{1+\delta}]$, which is $\leq \mathbb{E}[|Y_I|^{1+\delta}]$ for the associated interval $I$. Recall from Lemma \ref{lem:expectation} that $\mathbb{E}[|Y_I|^{1+\delta}]$ is at  most $C|I|^{1+\delta}$, and $|I| = 2^i < (\log(N))^d$. This means that for sufficiently large $N$, we can loosely bound $L'$ as:
\[L' = L - \frac{\mathbb{E}[|Z_j|^{1+\delta}]}{\delta L^{\delta}} > \frac{L}{2}.\]
We then have:
\[\mathbb{P}\left[ \left| \sum_{j=1}^L \tilde{Z}_j \right| \geq L'\right] \leq \frac{\mathbb{E}\left[ \left(\sum_{j=1}^L \tilde{Z}_j\right)^2\right]}{(L')^2} \leq \frac{4\mathbb{E}\left[\left(\sum_{j=1}^L \tilde{Z}_j\right)^2\right]}{L^2} \leq K' \mathbb{E}[|\tilde{Z}_j|^{1+\delta}]L^{-\delta},\]
for some constant $K'$ depending on $\delta$.
Combining this with (\ref{simplebound}), we have
\[\mathbb{P}\left[ \left| \sum_{j=1}^L Z_j - \sum_{j=1}^L \mathbb{E}[Z_j]\right| \geq L \right] \leq K'' \left( \max \left\{\mathbb{E}[|\tilde{Z}_j|^{1+\delta}], \mathbb{E}[|Z_j|^{1+\delta}] \right\}\right)L^{-\delta},\]
where $K'':= K'+1$.

Now we consider the quantity $\mathbb{E}[|\tilde{Z}_j|^{1+\delta}]$. We note that:
\[\mathbb{E}[|\tilde{Z}_j|^{1+\delta}] = \mathbb{E}\left[|\overline{Z}_j - \mathbb{E}[\overline{Z}_j]|^{1+\delta}\right].\] Since $\overline{Z}_j$ is a non-negative random variable, $|\overline{Z}_j - \mathbb{E}[\overline{Z}_j]| \leq \max \{\overline{Z}_j, \mathbb{E}[\overline{Z}_j]\}$. Then, \\ $\left(\max \{\overline{Z}_j, \mathbb{E}[\overline{Z}_j]\}\right)^{1+\delta} \leq \overline{Z}_j^{1+\delta} + \left(\mathbb{E}[\overline{Z}_j]\right)^{1+\delta}$. Since $g(x) = x^{1+\delta}$ is a convex function on $[0, \infty)$, Jensen's inequality implies that $\left(\mathbb{E}[\overline{Z}_j]\right)^{1+\delta} \leq \mathbb{E}[\overline{Z}_j^{1+\delta}]$. Therefore,
\[\mathbb{E}[|\tilde{Z}_j|^{1+\delta}] \leq 2\mathbb{E}[\overline{Z}_j^{1+\delta}] \leq 2\mathbb{E}[Z_j^{1+\delta}].\]

Since $Z_j \leq Y_I$ for the associated interval $I$, we have:
\[\mathbb{P}\left[ \left| \sum_{j=1}^L Z_j - \sum_{j=1}^L \mathbb{E}[Z_j]\right| \geq L \right] \leq 2K'' \mathbb{E}[|Y_I|^{1+\delta}]L^{-\delta}
\leq K |I|^{1+\delta} L^{-\delta},\]
for some constant $K$ depending on $\delta$, by Lemma \ref{lem:expectation}. Since $L := N/|I|$, we can rewrite this as $K |I|^{1+2\delta} N^{-\delta}$. Recalling that $|I| = 2^i$, we have:
\[\mathbb{P}\left[ \left| \sum_{j=1}^L Z_j - \sum_{j=1}^L \mathbb{E}[Z_j]\right| \geq L \right] \leq K (2^i)^{1+2\delta}N^{-\delta}.\]
\end{proof}

Now, we fix $N$ and consider summing these error probabilities in Lemma \ref{lem:errorprob} for all $i$ such that $2^i < (\log(N))^d$. We also fix a value $\delta'$ such that $0 < \delta' < \delta$. Since the number of such $i$ and all of the terms except $N^{-\delta}$ are polylogarithmic in $N$, we get that, for all $N$ sufficiently large with respect to $\delta$ (and $\delta, \delta'$):

\begin{equation}\label{lowlevels}
\sum_{i=0}^{d \log \log N} \mathbb{P}\left[ \left|\sum_{\substack{I \in \mathcal{F}\\ |I| = 2^i}} Y_I\mathbb{I}_{Bad,I} - \sum_{\substack{I \in \mathcal{F}\\ |I| = 2^i}}\mathbb{E}[Y_I \mathbb{I}_{Bad,I}] \right| \geq N\cdot 2^{-i} \right] \leq N^{-\delta'}.
\end{equation}

We will refer to the levels of $\mathcal{F}$ and $\mathcal{F}_s$ with $2^i < (\log N)^d$ as the ``low" levels. The left hand side of (\ref{lowlevels}) is an upper bound on the probability of the contribution of \emph{any} low level of $\mathcal{F}$ to the quantity (\ref{badcontribution}) exceeding its expectation by more that $N\cdot 2^{-i}$. The very same argument can be applied to the low levels of $\mathcal{F}_s$. Since we are considering only values of $N$ which are powers of 2, and
\[\sum_{n=1}^\infty 2^{-n\delta'} < \infty,\]
we can apply the Borel-Cantelli Lemma to conclude that almost surely, only finitely many values of $N$ will have a low level which contributes more than $N\cdot 2^{-i}$ plus its expected contribution.

When the contributions of all the low levels obey this bound, we have:
\begin{equation}\label{lowlevels2}
\sum_{\substack{I \in \mathcal{F} \cup \mathcal{F}_s\\ |I| < (\log N)^d}} Y_{I} \mathbb{I}_{Bad,I} \leq \sum_{i=0}^{d \log \log N} N \cdot 2^{-i} +
\sum_{\substack{I \in \mathcal{F}\cup \mathcal{F}_s \\ |I| < (\log N)^d}} \mathbb{E}[Y_I \mathbb{I}_{Bad,I}].
\end{equation}
We observe:
\[\sum_{i=0}^{d \log \log N} N \cdot 2^{-i} < N\sum_{i=0}^\infty 2^{-i} = 2N.\]

We will bound the second quantity using (\ref{expectedbad}). Recalling that $B$ was fixed so that $\frac{B}{576} \left(1-\frac{1}{1+\delta}\right) > 2$, (\ref{expectedbad}) implies:
\[\mathbb{E}[Y_I \mathbb{I}_{Bad,I}] \leq C''|I| \left(\frac{1}{\ln^2 N} + \frac{D'''}{|I|^{\sigma'}}\right),\]
where $C'', D''', d, \sigma'$ are constants depending on $\delta, \delta$. We note that this is merely a restatement of (\ref{finite3}). Thus, we have:
\begin{equation}\label{lowlevels3}
\sum_{\substack{I \in \mathcal{F}\cup \mathcal{F}_s \\ |I| < (\log N)^d}} \mathbb{E}[Y_I \mathbb{I}_{Bad,I}] \leq
2C''N \sum_{i=0}^{d\log \log N} \frac{1}{\ln^2 N} + \frac{D'''}{2^{i\sigma'}}.
\end{equation}
Next, we observe that
\[2C'' N \sum_{i=0}^{d \log \log N} \frac{1}{\ln^2 N} \leq C''' \frac{N (\ln \ln N)}{\ln^2 N},\]
for some constant $C'''$ depending on $\delta$.
Finally, we note that
\[ \sum_{i=0}^\infty \frac{1}{2^{i\sigma'}} < \infty.\]

Putting these results together with Lemma \ref{lem:bigintervals}, we have proven that quantity (\ref{badcontribution}) divided by $N\ln \ln N$ goes to zero as $N$ goes to infinity, for $N$'s which are powers of 2. To achieve a result for all values of $N$, we consider an $N$ which is an arbitrary positive integer, and let $N'$ the smallest power of 2 such that $N \leq N'$. Then $N' <2N$, and $N'\ln \ln N' < 3N\ln \ln N$ for instance (for all but very small $N$). Next, we claim that the contribution of the bad intervals to $\left|\left| \{X_i\}^N \right| \right|^2_{V^2}$ is bounded by quantity (\ref{badcontribution}) for $N'$. To see this, note that any bad interval in the maximal partition of $[N]$ will fall in an interval $I \in \mathcal{F} \cup \mathcal{F}_s$ for $N'$ (of size less than 4 times the size of the bad interval), and this $I$ must have $\mathbb{I}_{Bad,I}=1$ (recall that we defined these indicator variables to detect ``badness" with respect to $N'/2 < N$ for any subintervals of sufficient size). Thus, we have proven:

\begin{theorem}\label{thm:badintervals}
For a positive integer $N$, we let $\mathcal{B}_N$ denote the contribution of intervals $I$ such that $S_I^2 > B |I| \ln \ln N$ to the quantity $\left|\left| \{X_i\}^N \right| \right|^2_{V^2}$. Then:
\[\frac{\mathcal{B}_N}{N \ln \ln N} \rightarrow 0 \text{ a.s. as } N \rightarrow \infty.\]
\end{theorem}
This concludes our treatment of the bad intervals.

\subsection{The Medium Intervals}\label{sec:medium}
We first reduce to the case of bounded $X_i$'s (i.e. $|X_i| \leq M$ for some constant $M$). We consider the truncation of an unbounded $X_i$, denoted by $\overline{X}_i := X_i \mathbb{I}_{|X_i|\leq M}$. We define $Z_i := X_i - \overline{X}_i$. We then have:
\[\left|\left| \{X_i\}^N \right| \right|^2_{V^2} = \left|\left| \{(\overline{X_i}-\mathbb{E}[\overline{X}_i]) + (Z_i - \mathbb{E}[Z_i])\}^N \right|\right|^2_{V^2}.\]
By the triangle inequality for the $\ell^2$ norm, this is:
$\leq \left|\left|\{\overline{X}_i - \mathbb{E}[\overline{X}_i]\}^N \right| \right|_{V^2} + \left|\left| \{Z_i - \mathbb{E}[Z_i]\}^N\right|\right|_{V^2}.$

Now, $\frac{Z_i- \mathbb{E}[Z_i]}{\sqrt{Var[Z_i]}}$ are identically distributed random variables with expectation equal to 0, variance equal to 1, and finite $2(1+\delta)$ moment. We can therefore apply Theorem \ref{thm:badintervals} to conclude:
\[\limsup_{N \rightarrow \infty} \frac{\left|\left| \{Z_i-\mathbb{E}[Z_i]\}^N \right| \right|^2_{V^2}}{N \ln \ln N} \leq BVar[Z_i] \emph{ a.s.}\]
We recall that $B$ is a constant depending only on $\delta$. Since $\mathbb{E}[X_i^2] < \infty$, we can choose $M$ sufficiently large so that $B Var[Z_i]$ is much smaller than $\epsilon$, say $< \epsilon/2$ (and $Var[\overline{X}_i]$ is very close to 1). It then suffices to bound the contribution of the medium intervals for the variables
$\frac{\overline{X}_i - \mathbb{E}[\overline{X}_i]}{\sqrt{Var[\overline{X}_i]}}$ as $\leq 2(1+\epsilon/2) N \ln \ln N$. These variables are bounded, identically distributed, mean zero, variance 1, and have finite $2(1+\delta)$ moment.

To control the contribution of the medium intervals for bounded variables, we will first choose parameters $\epsilon_0, \epsilon_1 >0$ such that $\epsilon_0 < \epsilon$, $\epsilon_1^2 > \epsilon_0$, and:
\begin{equation}\label{epsilons}
(1+\epsilon)(1+\epsilon_0)^{-1}(1-\epsilon_1)^2 > 1.
\end{equation}
The reason for this constraint will become clear later. We also define $\epsilon'$ by $(1+\epsilon')^2 = 1+\epsilon_0$. We let $M$ denote the bound on the absolute values of the $X_i$'s. For a fixed positive integer $N$, we let $N' = (1+\epsilon')^n$ denote the real number which is the smallest integral power of $1+\epsilon'$ that is $\geq N$. We now define a family of real intervals which we will denote by $\mathcal{H}$. This will be similar to $\mathcal{F}\cup \mathcal{F}_s$, the family of intervals we considered in our analysis of the bad intervals, but now our interval lengths will be powers of $1+\epsilon'$ instead of powers of two.

$\mathcal{H}$ will be a union of several families of intervals, denoted $\mathcal{H}_0, \ldots, \mathcal{H}_h$, where $h$ is a function of $\epsilon'$. $\mathcal{H}_0$ consists of all intervals of the form:
\[ ((c-1)(1+\epsilon')^i, c(1+\epsilon')^i], \; i \in \{0, 1, \ldots, n\}, \; c \in \mathbb{Z} \cap [1, (1+\epsilon')^{n-i}].\]
For $j \neq 0$, $\mathcal{H}_j$ consists of all intervals of the form:
\[((c-1)(1+\epsilon')^i+j\epsilon'(1+\epsilon')^{i-1}, c(1+\epsilon')^i + j\epsilon'(1+\epsilon')^{i-1}], i \in \{1, \ldots, n-1\}, c \in \mathbb{Z} \cap [1, (1+\epsilon')^{n-i} -1].\] We will refer to intervals of size $(1+\epsilon)^i$ as belonging to level $i$ in each $\mathcal{H}_j$.
We set $h$ (the maximum value of $j$) to be $\frac{1+\epsilon'}{\epsilon'}-1$. Essentially, we are taking the intervals in $\mathcal{H}_0$ and shifting them by multiples of $\epsilon'(1+\epsilon')^{i-1}$ and stopping when we would reach the next interval. We now prove:

\begin{lemma}\label{lem:intervals2} Let $I' \subseteq [N]$ denote an arbitrary interval. There exists some interval $I \in \mathcal{H}$ such that $I' \subseteq I$ and $|I| < (1+\epsilon')^2|I'|$.
\end{lemma}

\begin{proof} We proceed similarly to the proof of Lemma \ref{lem:intervals}. We define the integer $i$ by the inequality $(1+\epsilon')^{i-1} < |I'| \leq (1+\epsilon')^i$. If $i \geq n-1$, then $I = (0, (1+\epsilon')^n]$ suffices. Otherwise, we consider intervals in $\mathcal{H}$ of size $(1+\epsilon')^{i+1}$. There exist integers $c,k$ such that the leftmost endpoint of $I$ is $> c(1+\epsilon')^{i+1} + k\epsilon'(1+\epsilon')^i$ and $\leq c(1+\epsilon')^{i+1} + (k+1)\epsilon'(1+\epsilon')^i$ and $c \in [0, (1+\epsilon')^{n-i-1}-1]$, $k < \frac{1+\epsilon'}{\epsilon}$. Since $|I'| \leq (1+ \epsilon')^i$, this implies that:
\[I'\subseteq (c(1+\epsilon')^{i+1} + k\epsilon'(1+\epsilon')^i, c(1+\epsilon')^{i+1} + (k+1)\epsilon'(1+\epsilon')^i+(1+\epsilon')^i].\]
We can rewrite the containing interval as:
\[I := (c(1+\epsilon')^{i+1} + k\epsilon'(1+\epsilon')^i, (c+1)(1+\epsilon')^{i+1} + k\epsilon'(1+\epsilon')^i] \in \mathcal{H}_k.\]
This $I$ satisfies $|I| = (1+\epsilon')^{i+1} < (1+\epsilon')^2 |I'|$, since $|I'| > (1+\epsilon')^{i-1}$.
\end{proof}

We let $N'' := (1+\epsilon')^{n-1}$. We will refer to an interval $I \in \mathcal{H}$ as ``medium" if \[\max_{\substack{I' \subseteq I\\ |I'| > (1+\epsilon_0)^{-1}|I|}} S_{I'}^2 > 2(1+\epsilon)(1+\epsilon_0)^{-1} |I| \ln \ln N''.\]
We note that $I \in \mathcal{H}$ can contain a subinterval $I'$ of size $|I'| > (1+\epsilon')^{-2}|I|$ that is medium (in the sense of (\ref{def:medium})) with respect to some $N'' < N \leq N'$ \emph{only} when $I$ satisfies this new condition for being ``medium." We let $\mathbb{I}_{I,Med}$ be the indicator variable for the event that $I \in \mathcal{H}$ is ``medium". Now, for any $N$ between $N''$ and $N'$, the total length of the medium intervals (in the sense of (\ref{def:medium})) in the maximal partition for $N$ is $\leq \sum_{I \in \mathcal{H}} |I|\cdot \mathbb{I}_{I,Med}$. We will upper bound this quantity. We begin by upper bounding $\mathbb{P}[\mathbb{I}_{I,Med} = 1]$, for which we employ the following lemma.

\begin{lemma}\label{lem:max} Let $\{X_{i}\}$ be a sequence of independent mean 0 random variables such that $|X_i| \leq M \; \forall i$. Then
\[\mathbb{P}\left[ \max_{1\leq l \leq L} \left| \sum_{i=1}^{l} X_{i} \right| > t \right] \leq C e^{\frac{-t^2/2}{\sum_{i=1}^L\mathbb{E}X_i^2 + Mt/3}}\]
for some absolute constant $C$.
\end{lemma}

While this statement is certainly not new (all of the essential ideas are contained in Hoeffding's paper \cite{H63}), we have been unable to locate a reference for this precise formulation, so we have included a proof in Appendix \ref{app:doob}. The key ingredient in obtaining this maximal form is Doob's maximal inequality for martingales.

\begin{corollary}\label{cor:max} Let $\{X_i\}$ be a sequence of independent mean 0 random variables with $Var[X_i]=1$ and $|X_i| \leq M$ for all $i$. Then
\[\mathbb{P}\left[ \max_{\substack{I' \subseteq [L] \\ |I'| > (1+\epsilon_0)^{-1}L}} \left| S_{I'} \right| > t \right] \leq C' e^{\frac{-t^2(1-\epsilon_1)^2/2}{L + Mt\epsilon_1/3\epsilon_0}}\]
for some absolute constant $C'$.
\end{corollary}

\begin{proof} We apply the union bound to conclude:
\[\mathbb{P}\left[ \max_{\substack{I' \subseteq [L] \\ |I'| > (1+\epsilon_0)^{-1}L}} \left| S_{I'} \right| > t \right] \leq \mathbb{P}\left[ \max_{l \leq L} \left| S_{\ell}\right| > (1-\epsilon_1 t)\right] + \mathbb{P}\left[ \left|S_{\epsilon_0L}\right| > \epsilon_1 t\right].\]
To see this, note that any subinterval $I'\subseteq [L]$ of size $> (1+\epsilon_0)^{-1}L$ must have its left endpoint be $< \left(1 - \frac{1}{1+\epsilon_0}\right)L < \epsilon_0 L$. By Lemma \ref{lem:max}, this is:
\[\leq C exp\left( -\frac{t^2 (1-\epsilon_1)^2}{2(L + Mt(1-\epsilon_1)/3)}\right) + Cexp\left( - \frac{t^2 \epsilon_1^2}{2(\epsilon_0L + M t \epsilon_1 /3)}\right).\]
Recalling that $\epsilon_1^2 > \epsilon_0$, this is:
\[ \leq 2C exp\left( - \frac{t^2(1-\epsilon_1)^2}{2(L + Mt\epsilon_1/3\epsilon_0)} \right).\]
\end{proof}

Since our variables $X_i$ are bounded in absolute value by $M$, the maximum of $S_{I'}^2$ for subintervals $I' \subseteq I$ is always bounded as $\leq (M |I|)^2$. Thus, for intervals $I$ which are too small with respect to $N''$, $I$ cannot possibly be medium. More specifically, $\mathbb{I}_{I,Med}$ can only be equal to 1 when $|I|> \frac{2(1+\epsilon)}{M^2(1+\epsilon_0)} \ln \ln N''$. Thus, we can assume $N$ (and hence $N''$) is sufficiently large so that applying Corollary \ref{cor:max} yields:
\begin{equation}\label{mediumprobability}
\mathbb{P}\left[ \mathbb{I}_{I,Med} = 1\right] \leq C'exp\left(-(1+\epsilon_2)\ln \ln N''\right) = C' (\ln N'')^{-(1+\epsilon_2)},
\end{equation}
for some positive $\epsilon_2$. To see this, note that we are applying the corollary with the value $t = \sqrt{2(1+\epsilon)(1+\epsilon_0)^{-1}|I|\ln \ln N''}$, and by (\ref{epsilons}), $(1+\epsilon)(1+\epsilon_0)^{-1}(1-\epsilon_1)^2 > 1$. We consider $N''$ large enough so that $L$ dominates the term $Mt\epsilon_1/3\epsilon_0$ in the denominator. Here, $L$ is the number of integers in the interval $I$, which is asymptotically equal to $|I|$ (the length of the real interval $I$).

We consider intervals in level $i$ of $\mathcal{H}_j$ (where $i$ is large enough so that these intervals can be medium). We let $k_{ij}$ denote the number of such intervals, and $k_{ij} \sim (1+\epsilon')^{n-i}$. We note that the indicator random variables of these intervals $\mathbb{I}_{I,Med}$ are independent, because the intervals are disjoint. We will let $\ell_{ij}$ denote the total length of all the medium intervals in level $i$ of $H_j$. By a Chernoff bound:
\begin{equation}\label{chernoff}
\mathbb{P}\left[ \sum_{\substack{I \in H_j \\ level i}} \mathbb{I}_{I,Med} \geq 2C'k_{ij}(\ln N'')^{-(1+\epsilon_2)}\right] \leq exp\left(-C'k_{ij}(\ln N'')^{-(1+\epsilon_2)}/3\right).
\end{equation}
When the event \[\sum_{\substack{I \in H_j \\ level i}} \mathbb{I}_{I,Med} \leq 2C'k_{ij}(\ln N'')^{-(1+\epsilon_2)}\] occurs, we have $\ell_{ij} \leq 2C'N'(\ln N'')^{-(1+\epsilon_2)}$. We will call this event $E_{ij}$.

Since $k_{ij} \sim (1+\epsilon')^{n-i}$ and $\ln N'' \sim n$, we can choose a constant $d$ large enough so that:
\[ \sum_{n} \sum_{i < n - d\ln n} exp\left(-C'k_{ij}(\ln N'')^{-(1+\epsilon_2)}/3\right) \leq \sum_{n} \sum_{i < n-d\ln n} exp\left( -C''(1+\epsilon')^{n-i} n^{-1-\epsilon_2}\right) < \infty.\] By the Borel-Cantelli Lemma, we then have that almost surely, \emph{all} of the events $E_{ij}$ (for $i < n-d\ln n$) occur when $n$ is sufficiently large (i.e. $N'$ is sufficiently large).

To address the medium intervals in $\mathcal{H}_j$ for levels with $i \geq n-d\ln n$, we note there are $d \ln n \sim d \ln \ln N'$ such levels, and by (\ref{mediumprobability}), the expected length of the medium intervals on each level is at most $C'N' (\ln N'')^{-(1+\epsilon_2)}$. Therefore, for each such $i$, by Markov's inequality (for $\epsilon_3 < \epsilon_2$):
\[\mathbb{P}\left[\ell_{ij} \geq C' N' (\ln N'')^{-\epsilon_2+\epsilon_3}\right] \leq \frac{1}{(\ln N'')^{1+\epsilon_3}}.\]
Now, the quantity $\frac{d\ln n}{(\ln N'')^{1+\epsilon_3}} \sim \frac{d \ln n}{(n \ln (1+\epsilon'))^{1+\epsilon_3}}$ converges when we sum over $n$. Hence, another application of the Borel-Cantelli Lemma tells us that, almost surely, for $n$ sufficiently large we will have $\ell_{ij} \leq C' N'(\ln N'')^{-\epsilon_2 + \epsilon_3}$ for all $i \geq n-d\ln n$.

Putting everything together, we have that almost surely, for sufficiently large $n$:
\[\sum_{i,j} \ell_{ij} \leq \frac{2C'h}{\ln (1+\epsilon')} N' \ln N' (\ln N'')^{-(1+\epsilon_2)} + C' h (d \ln n) N'(\ln N'')^{-\epsilon_2 + \epsilon_3}.  \]
Since $\ln N'' \sim \ln N'$ and $\ln n \sim \ln \ln N'$, we see that this entire quantity is $o(N)$. Thus, we have proven:

\begin{theorem}\label{mediumintervals} Almost surely, for sufficiently large $N$, the length of the intervals in the maximal partition for $N$ which are medium in the sense of (\ref{def:medium}) is $o(N)$.
\end{theorem}

This completes our proof of Theorem \ref{thm:ubound}, since the total contribution of the medium intervals is at most $B \ln \ln N$ times the length of the medium intervals, and the contribution of the good intervals is at most $(2+\epsilon)N \ln \ln N$.

\section{The Lower Bound}
We now prove the following theorem:

\begin{theorem}\label{thm:lowerbound} We let $X_1, X_2, \ldots $ denote independent, identically distributed random variables with $\mathbb{E}[X_i] = 0$, $Var[X_i] =1$, and $\mathbb{E}[|X_i|^{2(1+\delta})] < \infty$ for some $\delta>0$. Then, for every $\epsilon>0$,
\[\liminf_{N\rightarrow \infty} \frac{\left|\left| \{X_i\}^N \right| \right|^2_{V^2}}{N \ln \ln N} \geq 2(1-\epsilon) \text{ a.s.}\]
\end{theorem}

We first argue that it suffices to prove this theorem when the random variables $X_i$ are bounded (i.e. $|X_i| \leq M$ for some constant $M$). To see why, we again consider the truncation of an unbounded $X_i$, denoted by $\overline{X_i} := X_i \mathbb{I}_{|X_i|\leq M}$, and we define $Z_i := X_i - \overline{X}_i$. We fix values $\epsilon_1, \epsilon_2$ sufficiently small and a value of $M$ sufficiently high such that:
\begin{equation}\label{constraint}
\sqrt{2(1-\epsilon_2)Var[\overline{X}_i]} - \sqrt{2(1+\epsilon_1)Var[Z_i]} \geq \sqrt{2(1-\epsilon)}.
\end{equation}
There exists such a choice of $M$, $\epsilon_1$, and $\epsilon_2$ because $\mathbb{E}[X_i^2] < \infty$, so choosing $M$ sufficiently large will make $Var[\overline{X}_i]$ sufficiently close to 1 and $Var[Z_i]$ sufficiently close to 0.

Now, $\frac{Z_i- \mathbb{E}[Z_i]}{\sqrt{Var[Z_i]}}$ is a random variable with expectation equal to 0, variance equal to 1, and finite $2(1+\delta)$ moment. We can hence apply Theorem \ref{thm:ubound} with $\epsilon_1$ to conclude that:
\begin{equation}\label{z}
\limsup_{N \rightarrow \infty} \frac{\left|\left| \{Z_i-\mathbb{E}[Z_i]\}^N \right| \right|^2_{V^2}}{N \ln \ln N} \leq 2(1+\epsilon_1)Var[Z_i] \emph{ a.s.}
\end{equation}

We note that $\left|\left| \{\overline{X}_i - \mathbb{E}[\overline{X}_i]\}^N \right| \right|^2_{V^2}= \left|\left| \{X_i - (Z_i-\mathbb{E}[Z_i])\}^N \right| \right|^2_{V^2}.$ (Here, we have used the fact that $\mathbb{E}[X_i]=0$, so $\mathbb{E}[\overline{X}_i] = - \mathbb{E}[Z_i]$.)
Employing the triangle inequality, we have:

\begin{equation}\label{triangle}
\left|\left| \{X_i\}^N \right| \right|_{V^2} \geq \left|\left| \{\overline{X}_i - \mathbb{E}[\overline{X}_i]\}^N \right| \right|_{V^2} - \left|\left| \{Z_i- \mathbb{E}[Z_i]\}^N \right| \right|_{V^2}.
\end{equation}

If we prove Theorem \ref{thm:lowerbound} for bounded variables, we can apply it to the $\overline{X}_i - \mathbb{E}[\overline{X}_i]$'s with $\epsilon_2$. Note that these are mean zero random variables with finite variance and finite $2(1+\delta)$ moment, and we can divide them by $\sqrt{Var[\overline{X}_i]}$ to make them have variance equal to one. This gives us:
\[\liminf_{N\rightarrow \infty} \frac{\left|\left| \{\overline{X}_i\}^N \right| \right|_{V^2}}{\sqrt{N \ln \ln N}} \geq \sqrt{2(1-\epsilon_2)Var[\overline{X}_i]} \text{ a.s.}\]
Combining this with (\ref{triangle}), (\ref{z}), and (\ref{constraint}), we have that
\[\liminf_{N\rightarrow \infty} \frac{\left|\left| \{X_i\}^N \right| \right|^2_{V^2}}{N \ln \ln N} \geq 2(1-\epsilon) \text{ a.s.}.\]
Therefore, we may assume from this point forward that the $X_i$'s are bounded in absolute value by a constant $M$.

To prove the theorem for bounded variables $X_i$, we will use an inequality proven in \cite{G} and a general strategy motivated by \cite{T}. We begin by following the approach in \cite{G} for proving the lower bound portion of the law of the iterated logarithm for bounded random variables.
We fix a value $0 < \alpha < 1/2$ and a parameter $\epsilon_3$ whose value will be set later.

We consider the sequence of integers $m_1, m_2, \ldots$ defined by $m_k := s^k$, for some suitably large integer $s >1$. As usual, we let $S_{m_k}$ denote the sum $X_1 + \cdots + X_{m_k}$. We note the following inequality from \cite{G}, p. 282:

\begin{lemma}\label{lem:inequality} For every $\epsilon'>0$, when  $s$ is sufficiently large with respect to $\epsilon'$, there exists $0 < \gamma < 1$ such that for all sufficiently large $k$,
\[\mathbb{P}\left[S_{m_k} - S_{m_{k-1}} \geq (1-\epsilon')\sqrt{2(m_k-m_{k-1}) \ln \ln m_k}\right]\geq \frac{1}{k^{\gamma}}.\]
\end{lemma}

We fix a value of $s$ and a value $\epsilon_4$ such that:
\begin{equation}\label{constraint3}
(1-\epsilon_4)\sqrt{1-1/s} > \sqrt{1-\epsilon_3},
\end{equation}
and $s$ is sufficiently large with respect to $\epsilon_4$ to apply Lemma \ref{lem:inequality}.
We consider values of $N$ which are powers of $s$, i.e. $N = s^n$ for some integer $n$.

We now prove:
\begin{lemma}\label{errorprob2}For $N=s^n$ sufficiently large and for each fixed $j \in [N]$,
\begin{equation}\label{probtobound}
\mathbb{P}\left[ \sup_{1\leq i_1 \leq i_2 \leq N^{1-\alpha}} \frac{\left(S_{i_1+j}-S_j\right)^2 + \left(S_{i_2+j}-S_{i_1+j}\right)^2}{i_2} < 2(1-\epsilon_3)\ln\ln N\right] \leq exp(-c n^{1-\gamma}),
\end{equation}
for some constants $c, \gamma>0 $ independent of $n$ with $\gamma <1$.
\end{lemma}

\begin{proof}
We consider values of $i_2$ which are $=s^k$ for $n/2 \leq k\leq n(1-\alpha)$ (i.e. $i_2 = m_k$). For each such $k$, Lemma \ref{lem:inequality} implies that:
\begin{equation}\label{singlestep}
\mathbb{P}\left[ S_{j+m_k} - S_{j+m_{k-1}}> (1-\epsilon_4)\sqrt{2(m_k-m_{k-1})\ln \ln m_k}\right] \geq \frac{1}{k^{\gamma}}.
\end{equation}
Now, we suppose that for some $n/2 \leq k \leq n(1-\alpha)$, we have
\[
S_{j+m_k} - S_{j+m_{k-1}}> (1-\epsilon_4)\sqrt{2(m_k-m_{k-1})\ln \ln m_k}.
\]
We will call this event $E_k$, and we denote its complement by $\overline{E}_k$. When $E_k$ occurs, we consider $i_2 := m_k$ and $i_1 := m_{k-1}$. Note that $i_2 \leq N^{1-\alpha}$.
Then:
\[\left(S_{i_1+j}-S_j\right)^2 + \left(S_{i_2+j}-S_{i_1+j}\right)^2 \geq 2(1-\epsilon_4)^2(m_k-m_{k-1})\ln \ln m_k.\]

Using (\ref{constraint3}), $m_k -m_{k-1} = s^k (1-1/s)$, and $k\geq n/2$, this is $> 2(1-\epsilon_3)s^k \ln \ln s^k$. In fact, since $\ln \ln N = \ln \ln s^n$ and $\ln \ln s^{n/2} = \ln \ln s^n + \ln 1/2)$, for sufficiently large $N$ we have:
\[\left(S_{i_1+j}-S_j\right)^2 + \left(S_{i_2+j}-S_{i_1+j}\right)^2 \geq 2(1-\epsilon_3)i_2 \ln \ln N. \]

Therefore, the probability on the left hand side of (\ref{probtobound}) is at most the probability that the event $E_k$ fails to occur for all $n/2 \leq k \leq n(1-\alpha)$. Observe that these events for different $k$'s are independent, because they involve disjoint sets of the variables $X_i$. Thus, by Lemma \ref{lem:inequality},
\[ \mathbb{P}\left[\overline{E}_k \forall k \in [n/2, (1-\alpha)n]\right] \leq \prod_{k = \lceil n/2 \rceil}^{\lfloor n(1-\alpha)\rfloor } \left(1 - \frac{1}{k^\gamma}\right) \leq \prod_{k = \lceil n/2 \rceil}^{\lfloor n(1-\alpha)\rfloor } e^{-\frac{1}{k^\gamma}},\]
since for positive real numbers $x$, $1-x \leq e^{-x}$. This can be rewritten as:
\[exp\left( -\sum_{k = \lceil n/2\rceil}^{\lfloor n(1-\alpha)\rfloor} \frac{1}{k^\gamma}\right).\]

We observe that
\[\sum_{k = \lceil n/2\rceil}^{\lfloor n(1-\alpha)\rfloor} \frac{1}{k^\gamma} \geq \frac{n(1-\alpha-1/2)-2}{n^\gamma (1-\alpha)^\gamma}.\]
Thus, we can choose a constant $c$ independent of $n$ such that
\[\sum_{k = \lceil n/2\rceil}^{\lfloor n(1-\alpha)\rfloor} \frac{1}{k^\gamma} \geq cn^{1-\gamma}.\]
Therefore, we have shown that:
\[\mathbb{P}\left[ \sup_{1\leq i_1 \leq i_2 \leq N^{1-\alpha}} \frac{\left(S_{i_1+j}-S_j\right)^2 + \left(S_{i_2+j}-S_{i_1+j}\right)^2}{i_2} < 2(1-\epsilon_3)\ln\ln N\right] \leq exp(-cn^{1-\gamma}).\]
\end{proof}

Of course, we need to address \emph{all} values of $N$ and not just those that are powers of $s$. To allow us to handle general values of $N$, we introduce the following the family of intervals, which we denote by $\mathcal{L}$. We let $C$ denote a sufficiently large constant which is a power of $s$ (just how large it should be will be determined later). We define $\mathcal{L}$ as a union of $C$ different interval families, denoted by $\mathcal{L}_0, \ldots, \mathcal{L}_{C-1}$. $\mathcal{L}_0$ consists of all intervals of the form: $(1+s+\cdots + s^{k-1}, 1+s+ \cdots +s^k]$, where $k$ is a non-negative integer. We note that these intervals are disjoint, and that for each $k$, there is exactly one interval of size $s^k$.
More generally, $\mathcal{L}_i$ consists of all intervals of the form: \[\left(1+s+\cdots +s^{k-1}+ \frac{is^{k+1}}{C}, 1+ s+ \cdots + s^k + \frac{is^{k+1}}{C}\right]\] (for $k$'s large enough so that $C$ divides $s^{k+1}$ when $i\neq 0$). Each $\mathcal{L}_i$ is a union of disjoint intervals, one of size $s^k$ for each (large enough) $k$. To visualize these intervals, first consider $\mathcal{L}_0$ as an infinite stretch of intervals, starting with one of size 1, then one of size $s$, then one of size $s^2$, and so on, each beginning just where the previous one left off. Now imagine dividing each of the intervals in $\mathcal{L}_0$ into $C$ pieces of equal size. Then $\mathcal{L}_1$ can be thought of as a copy of $\mathcal{L}_0$ where the intervals have been shifted so that they now end at what used to be the end of the first piece of the next interval. In $\mathcal{L}_2$, they are shifted so that they end at what used to the end of the second piece, and so on. Note that these shifts will cause (relatively small) gaps between the intervals in $\mathcal{L}_i$ for $i >0$.

We let $I$ denote an interval in $\mathcal{L}$. For each $j \in I$, we let $A_j$ denote the event whose probability is bounded in (\ref{probtobound}) and $\overline{A}_j$ denote its complement.
We let $P_I$ denote the set of points $j \in I$ for which $A_j$ occurs. We let $E_I$ denote the event that $|P_I| > \epsilon_3 |I|$. We now show:

\begin{lemma}\label{lem:BorelCantelli} Almost surely, only finitely many of the events $E_I$ occur (for $I \in \mathcal{L}$).
\end{lemma}

\begin{proof} We consider an $I \in \mathcal{L}$ of size $s^k$. Then, for each $j \in I$,
\[\mathbb{P}[j \in P_I] \leq exp(-c k^{1-\gamma}),\]
by Lemma \ref{errorprob2}. Hence, $\mathbb{E}[P_I] \leq s^k exp(-c k^{1-\gamma})$. By Markov's inequality,
\[\mathbb{P}[E_I] \leq \frac{\mathbb{E}[|P_I|]}{\epsilon_3 s^k} \leq \frac{1}{\epsilon_3} exp(-c k^{1-\gamma}).\]
We note that for each $k$, there are at most $C$ intervals $\mathcal{L}$ of size $s^k$. Therefore:
\[\sum_{I \in \mathcal{L}} \mathbb{P}[E_I] \leq \frac{C}{\epsilon_3} \sum_{k=0}^\infty exp(-c k^{1-\gamma}) < \infty.\]
Hence, by the Borel-Cantelli Lemma, with probability one only finitely many of the events $E_I$ occur.
\end{proof}

Finally, we will prove Theorem \ref{thm:lowerbound} by considering a sufficiently large $N$ and defining the following partition of $[N]$. We begin by picking out a useful set of disjoint intervals in $\mathcal{L}$. We will denote this set by $S$, and we consider it initially to be empty. There is a unique positive integer $n_0$ such that $1+s+ \cdots +s^{n_0} \leq N < 1+s+\cdots +s^{n_0+1}$. There also exists an $i_0$ such that:
\[ 0 \leq N - \left(1+s +\cdots + s^{n_0} + \frac{i_0 s^{n_0+1}}{C}\right) \leq \frac{s^{n_0+1}}{C}.\]
We add to $S$ the interval of size $s^{n_0}$ in $\mathcal{L}$ which ends at $1+s + \cdots + s^{n_0} + \frac{i s^{n_0+1}}{C}$, and we call this interval $I_1$. Now, there exists some $n_1$ and $i_1$ such that the left endpoint of $I_1$ falls in the range:
\[\left[1+s+ \cdots + s^{n_1}+ \frac{i_1 s^{n_1+1}}{C}, 1+ s+\cdots + s^{n_1} + \frac{(i_1+1)s^{n_1 +1}}{C} \right).\]
We then define $I_2$ to be the interval of size $s^{n_1}$ in $\mathcal{L}$ which ends at $1+s+ \cdots + s^{n_1}+ \frac{i_1 s^{n_1+1}}{C}$. We then consider the left endpoint of $I_2$, and we define $I_3$ analogously, continuing on until we reach a point where the next interval we would like to use does not exist.

When we are finished, $S$ is a finite set of disjoint intervals covering most of the length from 1 to $N$. To see this, note that each time we add an interval of size $s^{n_\ell}$, we insert a gap of size at most $s^{n_\ell+1}/C$. For $C$ chosen sufficiently large with respect to $s$, the gap will be only a small proportion of the length of the interval being added. Thus, we lose only a small fraction of the length of $N$ to these gaps. Additionally, we can afford to ignore the length from 1 up to $1+s+ \cdots + s^{n_0/2}$, since this is $o(s^{n_0})$ and hence $o(N)$. Thus, at least a length of
\begin{equation}\label{length}
N\left(1 - \frac{s}{C}\right) - (1+s+\cdots + s^{n_0/2})
\end{equation}
is contained in intervals in $\mathcal{L}$ of size at least $s^{n_0/2}$.

We now define our partition of $N$, choosing our endpoints iteratively. We start at 0. (Our current position will always be the last element of the interval we just added to the partition, i.e. the last endpoint we choose.) When we are in a gap between intervals in $S$, we choose the next endpoint in our partition to be the end of the gap and we move to the first point of the next interval in $S$. When we are inside an interval in $S$, we let $j$ denote our current position, and $size_j$ denote the size of the interval in $S$. If the event $A_j$ occurs, we add an interval of size 1 to our partition and move to $j+1$. If $\overline{A}_j$ occurs and $j+size_j^{1-\alpha}$ is still within the current interval of $S$, we choose the $i_1$ and $i_2$ which maximize $(S_{i_1+j}-S_j)^2 + (S_{i_2+j} - S_{i_1+j})^2$ for $1\leq i_1\leq i_2 \leq size_j^{1-\alpha}$ as the next two endpoints and move to the point $i_2+j$. If $j+size_j^{1-\alpha}$ lies outside the current interval of $S$, we choose the next endpoint to the be the beginning of the next interval in $S$. We continue in this way until we reach $N$.

Now we must prove a lower bound (almost surely) for the sum of the squared partial sums over the intervals in this partition. We will ignore the gaps (which contribute an amount $\geq 0$), and consider the contributions of the partition pieces which lie inside intervals of $S$. Almost surely, only finitely many of the intervals $I$ in $S$ have $|P_I| > \epsilon_3 |I|$. We assume that $N$ is large enough so that all of the intervals $I \in S$ of size at least $s^{n_0/2}$ have $|P_I| \leq \epsilon_3 |I|$. Thus, the pieces of the partition in each of these intervals $I$ contributes at least
\[2(1-\epsilon_3)^2(|I|-|I|^{1-\alpha}) \ln \ln |I|\]
to the sum of the squared partial sums over the partition. For $|I| \geq s^{n_0/2}$, $\ln \ln |I| \sim \ln \ln N$. Since these intervals cover a length that is at least (\ref{length}), we can choose $\epsilon_3$ small enough with respect to $\epsilon$ and $C$ large enough with respect to $s$ to obtain a contribution that is $\geq (2-\epsilon) N \ln \ln N$. This completes our proof of Theorem \ref{thm:lowerbound}.

\section{A Result for Convergence in Probability}
We now prove the exact asymptotic holds for convergence in probability when the variables $X_i$ only satisfy $Var[X_i]=1$ and not any higher moment condition.

\begin{theorem}\label{thm:uboundinprob} We let $X_1, X_2, \ldots$ denote independent, identically distributed random variables with mean zero and variance equal to one. Then for every $\epsilon, \delta >0$,
\[\mathbb{P}\left[ \left|\left|\{X_i\}^N\right|\right|^2_{V^2} > 2(1+\epsilon)N \ln \ln N\right] < \delta\]
for all sufficiently large $N$, and similarly,
\[\mathbb{P}\left[\left|\left|\{X_i\}^N\right|\right|^2_{V^2} < 2(1-\epsilon)N \ln \ln N\right] < \delta\]
for all sufficiently large $N$.
\end{theorem}

\begin{proof} We will use the results of \cite{Q98} along with truncation, writing $X_i = \overline{X}_i + Z_i$, where $\overline{X}_i := \mathbb{I}_{|X_i|\leq M}$ and $M$ is chosen sufficiently large so that $Var[\overline{X}_i]$ is close to 1, and $\mathbb{E}[Z_i^2] =\epsilon' < \frac{1}{c'} \epsilon_1$, where $c'$ is the value from Theorem \ref{thm:Qianupper} and $\epsilon_1, \epsilon_2>0$ are chosen so that $\sqrt{\epsilon_1} + \sqrt{1+\epsilon_2} \leq \sqrt{1+\epsilon}$. Then, applying Theorem \ref{thm:Qianupper} to the mean zero variables $Z_i - \mathbb{E}[Z_i]$, we have (for all sufficiently large $N$):
\[\mathbb{P}\left[ \left|\left|\{Z_i-\mathbb{E}[Z_i]\}^N\right|\right|^2_{V^2} > \epsilon_1(2N\ln \ln N)\right] < \frac{\delta}{2}.\]
By Egorov's Theorem, convergence almost surely implies convergence in probability, so we can apply Theorem \ref{thm:ourtheorem} to the bounded variables $\overline{X}_i - \mathbb{E}[\overline{X}_i]$ to obtain (for all sufficiently large $N$):
\[\mathbb{P}\left[ \left|\left|\{\overline{X}_i-\mathbb{E}[\overline{X}_i]\}^N\right|\right|^2_{V^2}> (1+\epsilon_2)2N\ln \ln N\right] < \frac{\delta}{2}. \]
Since $\sqrt{\epsilon_1} + \sqrt{1+\epsilon_2} \leq \sqrt{1+\epsilon}$, the triangle inequality implies that:
\[\mathbb{P}\left[\left|\left|\{X_i\}^N\right|\right|^2_{V^2} > 2(1+\epsilon)N \ln \ln N\right] < \delta.\]
The second bound can be proven by an analogous argument applying Theorem \ref{thm:ourtheorem} to $\overline{X}_i$, Theorem \ref{thm:Qianupper} to $Z_i$, and employing $\left|\left|\{X_i\}^N\right|\right|^2_{V^2} \geq \left|\left|\{\overline{X}_i-\mathbb{E}[\overline{X}_i]\}^N\right|\right|^2_{V^2} - \left|\left|\{Z_i-\mathbb{E}[Z_i]\}^N\right|\right|^2_{V^2}$.
\end{proof}

\texttt{A. Lewko, Department of Computer Science, The University of Texas at Austin}

\textit{alewko@cs.utexas.edu}
\vspace*{0.5cm}

\texttt{M. Lewko, Department of Mathematics, The University of Texas at Austin}

\textit{mlewko@math.utexas.edu}

\appendix

\section{Proof of Lemma \ref{lem:max}}\label{app:doob}
We first note Doob's inequality:

\begin{lemma}(Doob's Inequality \cite{D}) Let $\{M_{i}\}_{i=1}^{L}$ be a submartingale taking non-negative real values. Then
\begin{equation}
\mathbb{P} \left[ \sup_{0\leq \ell \leq L} M_{\ell} \geq t \right] \leq \frac{\mathbb{E}\left[M_{L}\right] }{t}.
\end{equation}
\end{lemma}

We now prove Lemma \ref{lem:max}. Let $g(y) = 2 \sum_{l=2}^{\infty}\frac{y^{l-2} }{l!} = \frac{2(e^y - 1 -y)}{y^2}$.
Now
\[\mathbb{E} \left[ e^{h \sum_{i=1}^L X_i} \right] = \prod_{i=1}^{L} \mathbb{E}[ e^{h X_i}] = \prod_{i=1}^{L} \mathbb{E} \left[ \sum_{k=0}^{\infty} \frac{h^k X_{i}^k}{k!} \right] \]
\[= \prod_{i=1}^{L} \mathbb{E} \left[1 + h(X_i) + \frac{1}{2}h^2 X_i^2 g(h X_i) \right]\leq  \prod_{i=1}^{L}  \left( 1 + h\mathbb{E}[X_i] + \frac{1}{2}h^2 \mathbb{E} [X_i^2] g(h M) \right) \]
\[\leq e^{h\sum_{i=1}^L \mathbb{E}[X_i] + \frac{1}{2} h^2 \sum_{i=1}^L \mathbb{E}[X_{i}^2] g(hM)}
=  e^{\frac{1}{2}h^2 g(hM) \sum_{i=1}^{L}\mathbb{E}[X_{i}^2] } .\]

We have:

\[\mathbb{P}\left[ \sum_{i=1}^{L} X_i \geq t \right] = \mathbb{P} \left[  e^{h\sum_{i=1}^{L} X_i} \geq e^{ht} \right] \]

and, more generally:

\[\mathbb{P}\left[ \max_{1\leq \ell \leq L}  \sum_{i=1}^{\ell} X_i \geq t \right] = \mathbb{P} \left[ \max_{1\leq \ell \leq L} e^{h\sum_{i=1}^{\ell} X_i} \geq e^{ht} \right].\]

Since $M_{\ell} := e^{h\sum_{i=1}^{\ell} X_i}$ forms a submartingale sequence, Doob's inequality yields:

\[\mathbb{P}\left[ \max_{1\leq \ell \leq L}  \sum_{i=1}^{\ell} X_i \geq t \right] \leq e^{-h t} \mathbb{E} \left[ e^{h \sum_{i=1}^L X_i} \right] \leq e^{-h t}e^{\frac{1}{2}h^2 g(hM) \sum_{i=1}^{L}\mathbb{E}[X_{i}^2] }  . \]

Using the fact that $g(y) \leq \frac{1}{1-y/3}$ for $y < 3$ (which follows from the Taylor expansion given above), we have

\[\mathbb{P}\left[ \max_{1\leq \ell \leq L}  \sum_{i=1}^{\ell} X_i \geq t \right] \leq e^{\frac{\frac{1}{2}h^2 \sum_{i=1}^{L}\mathbb{E}[X_{i}^2]}{1- \frac{hM}{3}} - h t}\]
when $hM < 3$.

Taking $h = \frac{t}{\sum \mathbb{E}X_{i}^2 + Mt/3}$ (which satisfies $hM <3$), we have:

\[ exp\left(\frac{\frac{1}{2}h^2 \sum_{i=1}^{L}\mathbb{E}[X_{i}^2]}{1- \frac{hM}{3}} - h t\right) = \]
\[exp \left({\frac{\frac{1}{2} t^2 \sum_{i=1}^{L}\mathbb{E}[X_{i}^2]}{\left(\sum_{i=1}^{L}\mathbb{E}[X_{i}^2] + Mt /3\right)^2\left(1 - \frac{t M}{3\left(\sum_{i=1}^{L}\mathbb{E}[X_{i}^2] + Mt /3\right)}\right)}} - \frac{t^2}{\sum_{i=1}^{L}\mathbb{E}[X_{i}^2]+M t/3}\right)\]
\[ = exp \left( \frac{\frac{1}{2} t^2 \sum_{i=1}^{L}\mathbb{E}[X_{i}^2]}{\left(\sum_{i=1}^{L}\mathbb{E}[X_{i}^2] + M t/3\right) \sum_{i=1}^{L}\mathbb{E}[X_{i}^2]} - \frac{t^2}{\sum_{i=1}^{L}\mathbb{E}[X_{i}^2] + M t/3}\right)\]
\[ = exp\left(\frac{- \frac{1}{2} t^2}{\sum_{i=1}^{L}\mathbb{E}[X_{i}^2] + M t/3}\right).\]

This establishes that
\[\mathbb{P}\left[ \max_{1\leq \ell \leq L} \sum_{i=1}^{l} X_{i} > t \right] \leq e^{\frac{-t^2/2}{\sum_{i=1}^L\mathbb{E}[X_i^2] + Mt/3}}.\]
Applying this result to the random variables $\{Y_{i}\}$ where $Y_{i}=-X_{i}$, we have
\[\mathbb{P}\left[ \max_{1\leq \ell \leq L} \sum_{i=1}^{l} X_{i} < -t \right] \leq e^{\frac{-t^2/2}{\sum_{i=1}^L\mathbb{E}[X_i^2] + Mt/3}}.\]
Combining these, we obtain the desired estimate
\[\mathbb{P}\left[ \max_{1\leq \ell \leq L} \left| \sum_{i=1}^{l} X_{i} \right| > t \right] \leq 2 e^{\frac{-t^2/2}{\sum_{i=1}^L\mathbb{E}[X_i^2] + Mt/3}}.\]

\end{document}